\documentclass{amsart}

\usepackage{url}
\usepackage{amsmath}
\usepackage{amssymb}
\usepackage{bm}
\usepackage{mathtools}
\mathtoolsset{showonlyrefs=true}
\usepackage{amsthm}
\usepackage{graphicx}
\usepackage{upgreek}
\usepackage[labelformat=simple]{subcaption}
\usepackage[margin=3cm]{geometry}

\newcommand{\bR}{\mathbb{R}}
\newcommand{\bN}{\mathbb{N}}
\newcommand{\bu}{\mathbf{u}}
\newcommand{\bp}{\mathbf{p}}
\newcommand{\bv}{\mathbf{v}}
\newcommand{\bH}{\mathbf{H}}
\newcommand{\bP}{\mathbf{P}}
\newcommand{\bV}{\mathbf{V}}
\newcommand{\bkappa}{\bm{\upkappa}}
\newcommand{\btau}{\bm{\uptau}}
\DeclareMathOperator{\grad}{grad}
\newcommand{\rd}{\mathrm{d}}

\newtheorem{lemma}{Lemma}[section]
\theoremstyle{definition}
\newtheorem{definition}[lemma]{Definition}

\newtheorem{example}[lemma]{Example}
\newtheorem{scheme}[lemma]{Scheme}
\allowdisplaybreaks 

\begin{document}

\title[Dissipative scheme for gradient flows of planar curves]{
Energy dissipative numerical scheme for gradient flows of planar curves 
using discrete partial derivatives and B-spline curves
}
\author{Tomoya Kemmochi}

\address{Graduate School of Mathematical Sciences, The University of Tokyo,
3-8-1, Komaba, Meguro-ku, Tokyo, 153-8914, Japan.}
\email{kemmochi@ms.u-tokyo.ac.jp}
\urladdr{http://t-kemmochi.github.io/}

% \PACS{PACS code1 \and PACS code2 \and more}
%65M60  	Finite elements, Rayleigh-Ritz and Galerkin methods, finite methods
%35A35  	Theoretical approximation to solution
%	53C44  	Geometric evolution equations (mean curvature flow, Ricci flow, etc.)

\maketitle

\begin{abstract}
In this paper, we develop an energy dissipative numerical scheme for gradient flows of planar curves, such as the curvature flow and the elastic flow.
Our study presents a general framework for solving such equations.
To discretize time, we use a similar approach to the discrete partial derivative method, which is a structure-preserving method for the gradient flows of graphs.
For the approximation of curves, we use B-spline curves.
Owing to the smoothness of B-spline functions, we can directly address higher order derivatives.
In the last part of the paper, we consider some numerical examples of the elastic flow, which exhibit topology-changing solutions and more complicated evolution.
Videos illustrating our method are available on YouTube%
%\footnote{URL: \url{https://www.youtube.com/playlist?list=PLMF3dSqWEii6coWnwHHOh12mrtMmIXIi0}}
.
\end{abstract}

\section{Introduction}
\label{section:intro}

% ?????

In this paper, we consider numerical methods for the computation of the $L^2$-gradient flow of a planar curve:
\begin{equation}
\bu_t = -\grad (\bu), \quad t>0,
\label{eq:intro-gf}
\end{equation}
where $\bu$ is a time-dependent planar curve. The gradient flow \eqref{eq:intro-gf} is energy dissipative,
since
\begin{equation}
\frac{d}{dt} E[\bu] = - \int |\grad E(\bu)|^2 ds \le 0.
\end{equation}
Here, $E$ is an energy functional, and $\grad E$ is the Fr\'echet derivative with respect to the $L^2$-structure with line integral $ds$. Thus, the curvature flow (the curve shortening flow)
\begin{equation}
\bu_t = \bkappa
\label{eq:intro-cf}
\end{equation}
and the elastic flow (the Willmore flow)
\begin{equation}
\bu_t = -2 \varepsilon^2 \left( \nabla_s^2 \bkappa + \frac{1}{2} |\bkappa|^2\bkappa \right) + \bkappa
\label{eq:intro-ef}
\end{equation}
have energy functionals
\begin{equation}
E[\bu] = \int ds, \quad \text{and} \quad
E[\bu] = \varepsilon^2 \int |\bkappa|^2 ds + \int ds,
\end{equation}
where $\bkappa$ is the curvature vector, and $\nabla_s$ is the tangential derivative (Example~\ref{ex:gf}).
Note that the elastic flow is a fourth-order nonlinear evolution equation.
% ??????

We consider a dissipative numerical scheme for \eqref{eq:intro-gf}, that is, a scheme which has the discrete energy dissipative property $E[\bu_h^{n+1}] \le E[\bu_h^n]$ at each time step.
In general, a numerical method that retains a certain property for a target equation is called structure-preserving.
It is known that the numerical solutions obtained by these methods are not only physically realistic but also have the advantage of numerical stability (cf.~\cite{FurM11,HaiLW06}).
In particular, structure-preserving methods are suitable for computations over long time intervals.

% ????Willmore???
Here, we consider in more detail the energy dissipation for the gradient flow \eqref{eq:intro-gf}.
As is well known, the (classical) solution of the curvature flow \eqref{eq:intro-cf} blows up in finite time.
In particular, if the initial curve is self-crossing, a cusp appears. 
On the other hand, the elastic flow \eqref{eq:intro-ef}, which is a regularized version of \eqref{eq:intro-cf}, has a unique global solution \cite{DziKS02}.
Moreover, as $\varepsilon \downarrow 0$, the solution of \eqref{eq:intro-ef} converges to that of \eqref{eq:intro-cf} under conditions when the classical solution of \eqref{eq:intro-cf} (see~\cite{BelMN07}) exists.
Note that the elastic flow formally degenerates to the curvature flow as $\varepsilon \downarrow 0$.
Let $T$ be the time at which a cusp appears.
Then, we are interested in determining the limit of the elastic flow as $\varepsilon \downarrow 0$ after time $T$.
Therefore, the structure-preserving method is effective for investigating numerically the long-time behavior of the elastic flow with self-crossing initial curves.

% ???? ??????????, ???????.
There are many works that consider the numerical computation of equations \eqref{eq:intro-cf} and \eqref{eq:intro-ef} (for example,~\cite{BarGN07,BarGN10,DecD09,DecDE05,DziKS02}).
However, none of them explicitly consider the discrete energy dissipation property.
Although some numerical examples in these works seem to be dissipative, no mathematical proofs are given.
Moreover, all of them use the P1-finite element method, which approximates solutions using polygonal curves.
Therefore, since they cannot handle higher order derivatives directly, it is necessary to consider the mixed formulation for fourth-order equations.
It is remarkable that these studies present numerical schemes for individual problems, and do not give a general framework for approximating gradient flows \eqref{eq:intro-gf}.
There is a further issue regarding the arrangement of nodal points, when a curve evolution equation is discretized with polygonal curves.
However, we do not consider this issue in detail in this paper.

For gradient flows of graphs, there are some general frameworks to construct dissipative numerical schemes.
In \cite{Fur99,FurM11}, a structure-preserving finite difference scheme, the so-called discrete variational derivative method, is proposed to approximate the solution of the equation
\begin{equation}
u_t = (-1)^{s+1} \left( \frac{\partial}{\partial x} \right)^{2s} \frac{\delta E}{\delta u}
\label{eq:intro-gf2}
\end{equation}
over an interval in $\bR$.
Here, $\delta E/ \delta u$ is the Fr\'echet derivative of the energy functional $E$ with respect to the usual $L^2$-structure, i.e., the $L^2$ space with respect to the Lebesgue measure.
Finite element schemes, called discrete partial derivative methods (DPDM), for the same problems are presented in \cite{Mat08}.
The main idea of these studies is to discretize the chain rule. 
In \cite{Mat08}, discrete partial derivatives are introduced and a discretization of the chain rule is obtained.
A similar approach is used in our scheme.
See also \cite{AimMM15} for local discontinuous Galerkin schemes for the above problems.

% ???????????????? (isogeometric analysis; IGA)

In the present paper, we apply the idea of DPDM to the discretization of the time variable.
We obtain the discrete chain rule formula with respect to the $L^2(ds)$-structure.
In contrast to problems such as \eqref{eq:intro-gf2}, our problem \eqref{eq:intro-gf} is accompanied by the line element $ds$, which increases the complexity of the problem.
Due to the inclusion of the line element, we cannot use the DPDM for \eqref{eq:intro-gf}.
Therefore, we will present a generalization of the DPDM and derive a new scheme for weak forms of general gradient flows (Scheme~\ref{scheme:semi-disc}).
For the approximation of curves, we use B-spline curves (cf.~\cite{Far99,PieT97}).
The B-spline approach (also called NURBS in general) is widely used to compute the solution to large-scale deformation problems, such as fluid-structure interaction (see, e.g., \cite{BazTT13}).
This method is called isogeometric analysis~\cite{CorHB09}.
It is worth emphasizing that a B-spline curve of degree $p$ (Definition~\ref{def:B-spline}) is a $C^{p-1}$ curve.
Hence we can directly address higher order derivatives,
and we can derive the fully discretized scheme by the Galerkin method (Scheme~\ref{scheme:full-disc}).
This procedure is independent of the properties of the energy functional, and thus, our scheme gives a general framework for the approximation of gradient flows \eqref{eq:intro-gf}.
We can also obtain the discrete energy dissipation with our scheme (Lemma~\ref{lem:dissipation}).
In this study, we do not consider solvability or error estimates.

% ?????
This paper is structured as follows.
In Section~\ref{section:preliminary}, we present some necessary definitions and 
notation for gradient flows on planar curves (Subsection~\ref{subsec:gf}), and then illustrate the use of the DPDM (Subsection~\ref{subsec:dpdm}) and
formally define B-spline curves (Subsection~\ref{subsec:bspline}).
In Section~\ref{section:scheme}, we derive our energy dissipative scheme (Scheme~\ref{scheme:full-disc}) under the framework of the DPDM, and
introduce the discrete energy dissipation property (Lemma~\ref{lem:dissipation}).
Finally, we present some numerical examples of our scheme in Section~\ref{section:num_exp}, and 
compute the elastic flow \eqref{eq:intro-ef}.
A topology-changing solution and more complicated evolution are reported, which have not previously been shown in the literature.
Videos illustrating our method are available on YouTube%
\footnote{URL: \url{https://www.youtube.com/playlist?list=PLMF3dSqWEii6coWnwHHOh12mrtMmIXIi0}}.

\section{Preliminaries}
\label{section:preliminary}

\subsection{Geometric gradient flows for planar curves}
\label{subsec:gf}
In this subsection, we summarize the basic properties of 
the geometric gradient flows of planar curves.
Let $\bH^1_\pi = \{ \bu \in H^1(0,1; \bR^2) \mid \bu(0) = \bu(1) \}$
and $\bH^m_\pi = \{ \bu \in H^m(0,1; \bR^2) \mid \bu' \in \bH^{m-1}_\pi \}$ for $m \in \bN$, $m \ge 2$.
Here, the space $H^m(0,1; \bR^2) = (H^m(0,1))^2$ is the $m$-th order Sobolev space of $L^2$-type.
Note that the space $\bH^m_\pi$ is embedded into the space of the planar closed $C^{m-1}$-curves.
We define an energy functional $E \colon \bH^m_\pi \to \bR$ as
\begin{equation}
E[\bu] = \int F(\bu,\bu', \dots, \bu^{(m)}) ds,
\quad \bu \in \bH^m_\pi,
\label{eq:eng}
\end{equation}
where $F \colon (\bR^2)^{m+1} \to \bR$ is the energy density function, and $ds = ds(\bu)$ is the line element of the curve $\bu$.
Let $\grad E$ be the Fr\'echet derivative of the functional $E$ in the topology of $L^2(ds)$. That is,
\begin{equation}
\langle \grad E(\bu), \bv \rangle 
= \int \grad E(\bu) \cdot \bv ds(\bu).
\end{equation}
Then, the gradient flow for $E$ is represented by the following evolution equation:
\begin{equation}
\bu_t = - \grad E(\bu), \quad t>0.
\label{eq:gf}
\end{equation}
Examples of the gradient flow are now considered.

\begin{example}\label{ex:gf}
\begin{enumerate}
\item (Curvature flow) If $E[\bu] = \int ds$, then equation \eqref{eq:gf} is the curvature flow
\begin{equation}
\bu_t = \bkappa,
\label{eq:curvature}
\end{equation}
where $\bkappa = \bu_{ss}$ is the curvature vector, and $s$ is the arc-length parameter.
\item (Elastic flow) If $E[\bu] = \varepsilon^2 \int |\bkappa|^2 ds + \int ds$, then equation \eqref{eq:gf} is the elastic flow (or Willmore flow)
\begin{equation}
\bu_t = -2\varepsilon^2 \left( \nabla_s^2 \bkappa + \frac{1}{2}|\bkappa|^2 \bkappa \right) + \bkappa,
\label{eq:elastic}
\end{equation}
where $\nabla_s \bv = (\bv, \btau) \btau - \bv$, and $\btau = \bu_s$.
\end{enumerate}
\end{example}
In this paper, we focus on the energy dissipation property, which is given as follows:
\begin{equation}
\frac{d}{dt} E[\bu] = \int \grad E(\bu) \cdot \bu_t ds 
= - \int |\grad E(\bu)|^2 ds \le 0.
\end{equation}

\subsection{Discrete partial derivative method}
\label{subsec:dpdm}
In this subsection, we introduce the DPDM, which was first presented in \cite{Mat08}.
Let $E \colon H^2_\pi(0,1) \to \bR$ be an energy functional that is defined as
\begin{equation}
E[u] = \int_0^1 G(u_\zeta,u_{\zeta\zeta}) d\zeta , \quad u=u(\zeta) \in H^2_\pi(0,1),
\label{eq:eng-graph}
\end{equation}
where $G=G(p,q) \colon \bR^2 \to \bR$ is the energy density function.
Although we can consider more general energy functionals and density functions, we consider energy functionals $E$ such as \eqref{eq:eng-graph} for simplicity.
Let us denote the first variation of $E$ by $\delta E/\delta u$, i.e.,
\begin{equation}
\frac{\delta E}{\delta u} = -\frac{\partial}{\partial \zeta} G_p(u_\zeta,u_{\zeta\zeta}) + \left( \frac{\partial}{\partial \zeta} \right)^2 G_q(u_\zeta,u_{\zeta\zeta}).
\end{equation}
The $L^2$-gradient flow for the energy $E$ is the evolution equation
\begin{equation}
u_t = - \frac{\delta E}{\delta u}, \quad t>0.
\label{eq:gf-graph}
\end{equation}
This equation also has the energy dissipation property as in \eqref{eq:gf}. Indeed,
\begin{equation}
\frac{d}{dt} E[u] = \int_0^1 \frac{\delta E}{\delta u} u_t d\zeta
= - \int_0^1 \left| \frac{\delta E}{\delta u} \right|^2 d\zeta
\le 0.
\end{equation}

DPDM is an energy dissipative numerical scheme for the case of equation \eqref{eq:gf-graph}.
In DPDM, the discrete partial derivatives
$\partial G_\rd / \partial (u',v')$ and the others are defined as the functions that satisfy the following relation:
\begin{equation}
G(u_\zeta,u_{\zeta\zeta}) - G(v_\zeta,v_{\zeta\zeta})
= \frac{\partial G_\rd}{\partial(u_\zeta,v_\zeta)}(u_\zeta-v_\zeta)
+ \frac{\partial G_\rd}{\partial(u_{\zeta\zeta},v_{\zeta\zeta})}(u_{\zeta\zeta} - v_{\zeta\zeta}),
%\quad \forall u,v \in H^2_\pi(0,1).
\label{eq:dpd-graph}
\end{equation}
for all $u,v \in H^2_\pi(0,1)$.
Note that the partial derivatives which solve this relation may not be unique.
When $G$ is written in a certain form, a method for deriving the partial derivatives is given in \cite{Mat08}.
Next, we define the discrete analogue of the first variation as
\begin{equation}
\frac{\delta E_\rd}{\delta(u,v)} = 
-\frac{\partial}{\partial \zeta} \frac{\partial G_\rd}{\partial(u_\zeta,v_\zeta)} 
+ \left( \frac{\partial}{\partial \zeta} \right)^2 \frac{\partial G_\rd}{\partial(u_{\zeta\zeta},v_{\zeta\zeta})}.
\end{equation}
Note that this function satisfies the relation
\begin{equation}
E[u] - E[v] = \int_0^1 \frac{\delta E_\rd}{\delta(u,v)} (u-v) d\zeta.
\label{eq:disc-chain-graph}
\end{equation}
Therefore, we can derive the discretized equation 
\begin{equation}
\frac{u^{n+1} - u^n}{\Delta t} = 
- \frac{\delta E_\rd}{\delta(u^{n+1},u^n)},
\quad n \in \bN.
\label{eq:dpdm-graph-strong}
\end{equation}
Following \cite{Mat08}, we can derive an energy dissipative time-discretization for \eqref{eq:gf-graph} as follows.
The following scheme is the weak form of equation \eqref{eq:dpdm-graph-strong}.
\begin{scheme}[DPDM for the gradient flow of a graph]
Let $n \in \bN$ and $u^n \in H^2_\pi(0,1)$ be given.
Find $u^{n+1} \in H^2_\pi(0,1)$ that satisfies
\begin{equation}
\left( \frac{u^{n+1} - u^n}{\Delta t}, v \right) 
= - \left( \frac{\partial G_\rd}{\partial(u^{n+1}_\zeta,u^n_\zeta)}, v_\zeta \right)
- \left( \frac{\partial G_\rd}{\partial(u^{n+1}_{\zeta\zeta},u^n_{\zeta\zeta})}, v_{\zeta\zeta} \right),
%\quad \forall v \in H^2_\pi(0,1),
\label{eq:dpdm-graph}
\end{equation}
for all $v \in H^2_\pi(0,1)$.
%for $n \in \bN$ and given $u^n \in H^2_\pi(0,1)$.
\end{scheme}

We can now check the discrete energy dissipation property with the following lemma.
Here we give the proof for comparison with our scheme (see Lemma~\ref{lem:dissipation}).
\begin{lemma}
Let $u^n$ and $u^{n+1}$ satisfy the relation \eqref{eq:dpdm-graph}.
Then, we have
\begin{equation}
\frac{E[u^{n+1}] - E[u^n]}{\Delta t} 
= - \left\| \frac{u^{n+1} - u^n}{\Delta t} \right\|_{L^2(0,1)}^2
\le 0
\end{equation}
for all $n \in \bN$ and $\Delta t > 0$.
\end{lemma}

\begin{proof}
Let us write
\begin{equation}
\partial_\rd u^n = \frac{u^{n+1} - u^n}{\Delta t}.
\end{equation}
Then, substituting $v = \partial_\rd u^n$ into the weak form, \eqref{eq:dpdm-graph}, we can derive
\begin{align}
\frac{E[u^{n+1}] - E[u^n]}{\Delta t}
&= \int_0^1 \frac{G(u^{n+1}_\zeta, u^{n+1}_{\zeta\zeta}) - G(u^{n}_\zeta, u^{n}_{\zeta\zeta})}{\Delta t} d\zeta \\
&= \left( \frac{\partial G_\rd}{\partial(u^{n+1}_\zeta,u^n_\zeta)}, \partial_\rd u^n_\zeta \right)
+ \left( \frac{\partial G_\rd}{\partial(u^{n+1}_{\zeta\zeta},u^n_{\zeta\zeta})}, \partial_\rd u^n_{\zeta\zeta} \right) \\
&= - \|\partial_\rd u^n\|_{L^2(0,1)}^2.
\end{align}
\end{proof}
Note that the key point is substituting $v = \partial_\rd u^n$ into \eqref{eq:dpdm-graph}.
Therefore, this proof can be derived also in the case of the Galerkin method.

\subsection{B-spline curves}
\label{subsec:bspline}
In our scheme, we use B-spline curves to discretize the solution curves.
We say that a set of points $\Xi = \{ \xi_1, \xi_2, \dots, \xi_n \} \subset \bR$ is a \textit{knot vector} if $\xi_i \le \xi_{i+1}$ for all $i$.
\begin{definition}[B-spline basis functions and B-spline curves]
Let $p \in \bN_0$, $n \in \bN$, and $\Xi = \{ \xi_1, \xi_2, \dots, \xi_n \}$ be a knot vector.
\begin{enumerate}
\item The \textit{$i$-th B-spline basis function of degree $p$} with respect to $\Xi$ is a piecewise polynomial function $N^\Xi_{p,i}$ that is generated by the following formula:
\begin{align}
N^\Xi_{0,i}(\xi) &= \chi_{[\xi_i,\xi_{i+1})}(\xi), 
\quad \xi \in \bR,
%& i &= 1,2,\dots, n-1, 
\\
\shortintertext{for $i = 1,2,\dots, n-1$, and}
N^\Xi_{p,i}(\xi) &= \frac{\xi - \xi_i}{\xi_{i+p} - \xi_i} N^\Xi_{p-1,i}(\xi) + \frac{\xi_{i+p+1} - \xi}{\xi_{i+p+1} - \xi_{i+1}} N^\Xi_{p-1,i+1}(\xi), 
\quad \xi \in \bR,
%& i &= 1,2,\dots,n-p-1, \quad p \ge 1,
\end{align}
for $i = 1,2,\dots,n-p-1$ and $p \ge 1$, 
where $\chi_I$ is the characteristic function of $I \subset \bR$.
Here, if $\xi_{i+p}=\xi_i$ (resp.\ $\xi_{i+p+1}=\xi_{i+1}$), then the term $(\xi - \xi_i)/(\xi_{i+p} - \xi_i)$ (resp.\ $(\xi_{i+p+1} - \xi)/(\xi_{i+p+1} - \xi_{i+1})$) is null.
\item A curve $\bu \colon [a,b] \to \bR^2$ is a \textit{B-spline curve of degree $p$} if $\bu$ is represented by
\begin{equation}
\bu(\zeta) = \sum_{i=1}^{n-p-1} N^\Xi_{p,i}(\zeta) \bP_i,
\quad \zeta \in [a,b],
\end{equation}
for some knot vector $\Xi$ and $n \in \bN$.
The coefficient $\bP_i$ is called a \textit{control point}.
\end{enumerate}
\end{definition}
In fact, if the knot vector is disjoint (i.e., $i \ne j \implies \xi_i \ne \xi_j$), then it is known that $N^\Xi_{p,i}$ is a $C^{p-1}$-function.
For more details on the properties of B-spline functions, we refer the reader to \cite{Far99,PieT97}.

In the present paper, we only consider the periodic B-spline functions and curves.
Let $[a,b] \subset \bR$ be an interval, $p \in \bN$, $N \in \bN$, and $h=1/N$.
We define a knot vector $\Xi$ as
\begin{equation}
\Xi = \{ \xi_i \}_{i=1}^{N+2p+1} = \{ a-ph, a-(p-1)h, \dots, b+(p-1)h, b+ph \},
\label{eq:periodic-knot}
\end{equation}
and let $N^\Xi_{p,i}$ be the corresponding B-spline basis function.
Note that $\operatorname{supp} N^\Xi_{p,i} = [\xi_i,\xi_{i+p+1}]$.
Then, if $N > p$, we can see that
\begin{equation}
\left( \frac{d}{d\zeta} \right)^m N^\Xi_{p,i}(a) 
= 
\left( \frac{d}{d\zeta} \right)^m N^\Xi_{p,i+N}(b),
%\qquad i=1,2,\dots,p, \quad m=0,1,\dots,p-1,
\end{equation}
for $i=1,2,\dots,p$ and $m=0,1,\dots,p-1$.
Therefore, the function
\begin{equation}
B_{p,i}(\zeta) = B_{h,p,i}(\zeta) = \begin{cases}
N^\Xi_{p,i}(\zeta), & \zeta \in [a,\xi_{i+p+1}], \\
N^\Xi_{p,i+N}(\zeta), & \zeta \in [\xi_{i+N},b], \\
0, & \text{otherwise}
\end{cases}
\quad i=1,2,\dots,p
\label{eq:periodic-B-spline}
\end{equation}
is a periodic $C^{p-1}$-function in $[a,b]$.
The restriction $N^\Xi_{p,i}|_{[a,b]}$ for $i > p$ is also $C^{p-1}$-periodic on $[a,b]$.
Then, we define a closed B-spline curve as follows.

\begin{definition}[Periodic B-spline]
\label{def:B-spline}
Let $[a,b] \subset \bR$ be an interval, $p \in \bN$, $N \in \bN$ with $N>p$, and $h=1/N$.
Then, we define a \textit{periodic B-spline basis function of degree p} $B_{p,i} = B_{h,p,i}$ by \eqref{eq:periodic-B-spline} for $i \le p$ and by $B_{p,i} = N^\Xi_{p,i}|_{[a,b]}$ for $i =p,p+1,\dots,N$,
where $\Xi$ is a knot vector defined by \eqref{eq:periodic-knot}.
We also define a \textit{closed B-spline curve} as a curve $\bu \colon [a,b] \to \bR^2$ expressed by
\begin{equation}
\bu(\zeta) = \sum_{i=1}^{N} B_{p,i}(\zeta) \bP_i,
\quad \zeta \in [a,b],
\end{equation}
for some $\{ \bP_i \}_{i=1}^N \subset \bR^2$.
\end{definition}
It is clear that a closed B-spline curve is a $C^{p-1}$-curve.

\section{Derivation of an energy dissipative numerical scheme}
\label{section:scheme}

In this section, we derive a numerical scheme for geometric gradient flows \eqref{eq:gf} for the energy functional $E$ given by \eqref{eq:eng}.
We first consider the time discretization,
and recall the idea of the DPDM.
Using a similar approach as for the DPDM, we derive a discretization of the chain rule.
The definition of the partial derivatives \eqref{eq:dpd-graph} is a discrete analogue of the chain rule formula
\begin{equation}
\partial_t G(u_\zeta, u_{\zeta\zeta})
= G_p(u_\zeta, u_{\zeta\zeta}) u_{\zeta t} + G_q(u_\zeta, u_{\zeta\zeta}) u_{\zeta\zeta t},
\end{equation}
for a smooth function $u$.
In our case, the corresponding chain rule can be expressed as
\begin{equation}
\frac{d}{dt} E[\bu] = \int \grad E(\bu) \cdot \bu_t ds(\bu).
\label{eq:chain}
\end{equation}
Here, we denote the line element of $\bu$ by $ds(\bu)$ to emphasize the dependence on $\bu$.
Now, we discretize the chain rule \eqref{eq:chain}.
We first change the time derivatives to time differences by expressing
$\frac{d}{dt} E[\bu]$ and $\bu_t$ as $E[\bu] - E[\bv]$ and $\bu - \bv$, respectively.
Moreover, the line element $ds(\bu)$ should be changed appropriately.
In the original formula \eqref{eq:chain}, there is one function $\bu$ only.
However, in the discretization, there are two functions $\bu$ and $\bv$ as in \eqref{eq:dpd-graph}.
Therefore, we have some choices to discretize the term $ds(\bu)$,
for example, $ds(\bu)$, $ds(\bv)$, and $ds((\bu+\bv)/2)$.
Here, we use $ds((\bu+\bv)/2)$.
Then, we define a discrete gradient, $\grad_\rd E \colon (\bH^m_\pi)^2 \to \bR^2$, with a function that satisfies the following formula:
\begin{equation}
E[\bu] - E[\bv] = \int \grad_\rd E(\bu, \bv) ds\left( \frac{\bu+\bv}{2} \right),
\quad
\forall \bu,\bv \in \bH^m_\pi.
\label{eq:disc-chain}
\end{equation}
Thus, according to \eqref{eq:dpdm-graph-strong}, the strong form of the time-discrete problem is written as follows:
\begin{equation}
\frac{\bu^{n+1} - \bu^n}{\Delta t} 
= - \grad_\rd E(\bu^{n+1}, \bu^n),
\quad n \in \bN.
\label{eq:disc-gf-strong}
\end{equation}
The discrete chain rule \eqref{eq:disc-chain} can then be expressed as
\begin{equation}
E[\bu] - E[\bv] = \int_0^1 \left| \frac{\bu_\zeta+\bv_\zeta}{2} \right| \grad_\rd E(\bu, \bv)
\cdot (\bu - \bv) d\zeta,
\quad
\forall \bu,\bv \in \bH^m_\pi,
\label{eq:disc-chain2}
\end{equation}
and comparing \eqref{eq:disc-chain2} with \eqref{eq:disc-chain-graph}, we can derive the relationship between $\grad_\rd E$ and the discrete first derivative $\partial E_\rd / \partial (\bu,\bv)$ as follows.
\begin{equation}
\left| \frac{\bu_\zeta+\bv_\zeta}{2} \right| \grad_\rd E(\bu, \bv)
= \frac{\delta E_\rd}{\delta (\bu,\bv)}.
\end{equation}
Here $\partial E_\rd / \partial (\bu,\bv)$ is a vector-valued function.
Letting $G(\bp_0,\bp_1,\dots,\bp_m) := F(\bp_0,\bp_1,\dots,\bp_m) |\bp_1|$, the energy $E$ is expressed by
\begin{equation}
E[\bu] = \int_0^1 G(\bu, \bu_\zeta, \dots, \partial_\zeta^m \bu) d\zeta,
\end{equation}
and thus the discrete first derivative is given by
\begin{equation}
\frac{\delta E_\rd}{\delta (\bu,\bv)}
=
\sum_{j=0}^{m} (-1)^j \left( \frac{\partial}{\partial\zeta} \right)^j \frac{\partial G_\rd}{\partial(\partial_\zeta^j \bu, \partial_\zeta^j \bv)}.
\end{equation}
Here, we define the (vector-valued) partial derivatives,
\begin{equation}
\frac{\partial G_\rd}{\partial(\partial_\zeta^j \bu, \partial_\zeta^j \bv)}
=
\left(
\frac{\partial G_\rd}{\partial(\partial_\zeta^j u_1, \partial_\zeta^j v_1)}
,\
\frac{\partial G_\rd}{\partial(\partial_\zeta^j u_2, \partial_\zeta^j v_2)}
\right)^T,
\quad j=0,1,\dots,m
\end{equation}
as functions that satisfy the relation
\begin{equation}
G(\bu, \bu_\zeta, \dots, \partial_\zeta^m \bu) - G(\bv, \bv_\zeta, \dots, \partial_\zeta^m \bv)
= \sum_{j=0}^m \frac{\partial G_\rd}{\partial(\partial_\zeta^j \bu, \partial_\zeta^j \bv)} \cdot \partial_\zeta^j (\bu - \bv),
%\quad \forall \bu,\bv \in \bH^m_\pi.
\label{eq:dpm}
\end{equation}
for all $\bu,\bv \in \bH^m_\pi$.
Note that, as in the previous case (Subsection~\ref{subsec:dpdm}), the partial derivative may not be unique.
Now, instead of solving \eqref{eq:disc-gf-strong}, we may solve the equation
\begin{equation}
\left| \frac{\bu^{n+1}_\zeta + \bu^n_\zeta}{2} \right| \frac{\bu^{n+1} - \bu^n}{\Delta t}
= - \sum_{j=0}^{m} (-1)^j \left( \frac{\partial}{\partial\zeta} \right)^j \frac{\partial G_\rd}{\partial(\partial_\zeta^j \bu^{n+1}, \partial_\zeta^j \bu^n)},
\label{eq:disc-gf-strong2}
\end{equation}
and the weak form of \eqref{eq:disc-gf-strong2} gives our semi-discrete scheme for the gradient flow \eqref{eq:gf}.
Note that the time increment $\Delta t$ can differ at each step.

\begin{scheme}[Semi-discrete scheme for the geometric gradient flow]
\label{scheme:semi-disc}
Find $\bu^{n+1} \in \bH^m_\pi$ that satisfies
\begin{equation}
\left( \left| \frac{\bu^{n+1}_\zeta + \bu^n_\zeta}{2} \right| \frac{\bu^{n+1} - \bu^n}{\Delta t_n}, \bv \right)
=
- \sum_{j=0}^{m} \left( \frac{\partial G_\rd}{\partial(\partial_\zeta^j \bu^{n+1}, \partial_\zeta^j \bu^n)}, \partial_\zeta^j \bv \right),
\quad \forall \bv \in \bH^m_\pi
\label{eq:disc-gf-semi}
\end{equation}
for given $\bu^n \in \bH^m_\pi$.
\end{scheme}

We now consider the full discretization of the gradient flow \eqref{eq:gf}.
Let $\bV_h^p$ be the space of closed B-spline curves of degree $p$ as defined in Definition~\ref{def:B-spline}.
Then, by the Sobolev embedding theorem, $\bV_h^p \hookrightarrow \bH^m_\pi$ if $p \ge m+1$.
Thus, we can derive a fully discretized problem by the Galerkin method.
\begin{scheme}[Fully discretized scheme for the geometric gradient flow]
\label{scheme:full-disc}
Let $N \in \bN$, $h=1/N$, and $p \ge m+1$.
Assume $\bu_h^n \in \bV_h^p$ is given.
Find $\bu_h^{n+1} \in \bV_h^m$ that satisfies
\begin{equation}
\left( \left| \frac{\bu^{n+1}_{h,\zeta} + \bu^n_{h,\zeta}}{2} \right| \frac{\bu_h^{n+1} - \bu_h^n}{\Delta t_n}, \bv_h \right)
=
- \sum_{j=0}^{m} \left( \frac{\partial G_\rd}{\partial(\partial_\zeta^j \bu_h^{n+1}, \partial_\zeta^j \bu_h^n)}, \partial_\zeta^j \bv_h \right),
%\quad \forall \bv_h \in \bV_h^p
\label{eq:disc-gf}
\end{equation}
for all $\bv_h \in \bV_h^p$.
%for given $\bu_h^n \in \bV_h^p$.
\end{scheme}

Then, we can establish the discrete energy dissipation property.
\begin{lemma}[Discrete energy dissipation]
\label{lem:dissipation}
Let $\bu_h^n$ and $\bu_h^{n+1}$ satisfy the relation \eqref{eq:disc-gf}.
Then, we have
\begin{equation}
\frac{E[\bu_h^{n+1}] - E[\bu_h^n]}{\Delta t_n}
= - \int_0^1 \left| \frac{\bu^{n+1}_{h,\zeta} + \bu^n_{h,\zeta}}{2} \right| \left| \frac{\bu_h^{n+1} - \bu_h^n}{\Delta t_n} \right|^2 d\zeta
\le 0.
\end{equation}
\end{lemma}

\begin{proof}
Substituting $\bv_h = (\bu_h^{n+1} - \bu_h^n)/\Delta t_n$ into the scheme \eqref{eq:disc-gf}, we have
\begin{align}
&\int_0^1 \left| \frac{\bu^{n+1}_{h,\zeta} + \bu^n_{h,\zeta}}{2} \right| \left| \frac{\bu_h^{n+1} - \bu_h^n}{\Delta t_n} \right|^2 d\zeta \\
= & - \sum_{j=0}^{m} \left( \frac{\partial G_\rd}{\partial(\partial_\zeta^j \bu_h^{n+1}, \partial_\zeta^j \bu_h^n)}, \frac{\partial_\zeta^j (\bu_h^{n+1} - \bu_h^n)}{\Delta t_n} \right) \\
= & - \int_0^1 \frac{G(\bu_h^{n+1},\dots,\partial_\zeta^m \bu_h^{n+1}) - G(\bu_h^{n},\dots,\partial_\zeta^m \bu_h^{n})}{\Delta t_n} d\zeta \\
= & - \frac{E[\bu_h^{n+1}] - E[\bu_h^n]}{\Delta t_n},
\end{align}
by the definition of the partial derivatives \eqref{eq:dpm}.
Hence we can establish the desired assertion.
\end{proof}

\section{Numerical examples}
\label{section:num_exp}

In this section, we show some numerical examples of the elastic flow \eqref{eq:elastic} computed by our scheme \eqref{eq:disc-gf}.
Here, the functional is the elastic energy
\begin{equation}
E[\bu] = \varepsilon^2 \int |\bkappa|^2 ds + \int ds
= \int_0^1 \left( \varepsilon^2 \frac{\det(\bu_\zeta,\bu_{\zeta\zeta})^2}{|\bu_\zeta|^5} + |\bu_\zeta| \right) d\zeta,
\label{eq:elastic-energy}
\end{equation}
where
\begin{equation}
\det(\bu_\zeta,\bu_{\zeta\zeta})
= \det \begin{pmatrix}
u_{1,\zeta} & u_{1,\zeta\zeta} \\ u_{2,\zeta} & u_{w,\zeta\zeta}
\end{pmatrix}
= u_{1,\zeta}u_{2,\zeta\zeta} - u_{1,\zeta\zeta}u_{2,\zeta},
\quad 
%\bu(\zeta) = (u_1(\zeta), u_2(\zeta))^T.
\bu(\zeta) = \begin{pmatrix}
u_1(\zeta) \\ u_2(\zeta)
\end{pmatrix}.
\end{equation}
It is known that equation \eqref{eq:elastic} has a unique time-global solution (see, e.g., \cite[Theorem~3.2]{DziKS02}).
Therefore, the turning number $|\int \bkappa ds|/(2\pi) \in \bN$ is invariant.

To calculate the discrete partial derivatives for $E$,
let 
\begin{equation}
G_1(\bu_\zeta,\bu_{\zeta\zeta}) = \frac{\det(\bu_\zeta,\bu_{\zeta\zeta})^2}{|\bu_\zeta|^5},
\quad
G_2(\bu_\zeta) = |\bu_\zeta|
\end{equation}
for $\bu \in \bH^2_\pi$.
Then, the energy density function for $E$ is $G:=G_1+G_2$.
We can compute the partial derivatives of $G_2$ since
\begin{equation}
G_2(\bu_\zeta) - G_2(\bv_\zeta)
= \frac{\bu_\zeta + \bv_\zeta}{|\bu_\zeta| + |\bv_\zeta|} \cdot (\bu_\zeta - \bv_\zeta),
\end{equation}
which implies
\begin{equation}
\frac{\partial G_{2,\rd}}{\partial(u_{j,\zeta},v_{j,\zeta})}
= \frac{u_{j,\zeta} + v_{j,\zeta}}{|\bu_\zeta| + |\bv_\zeta|},
\quad j=1,2.
\end{equation}
We can derive the partial derivatives of $G_1$ in several ways.
In the following examples, these derivatives are computed by dividing $G_1(\bu) - G_1(\bv)$ as follows:
\begin{align}
G_1(\bu_\zeta,\bu_{\zeta\zeta}) - G_1(\bv_\zeta,\bv_{\zeta\zeta})
&= \frac{\det(\bu_\zeta, \bu_{\zeta\zeta})^2 - \det(\bv_\zeta, \bv_{\zeta\zeta})^2}{|\bu_\zeta|^5}  \\
&+ \det(\bv_\zeta, \bv_{\zeta\zeta})^2 \left(
\frac{1}{|\bu_\zeta|^5} - \frac{1}{|\bv_\zeta|^5}
\right).
\end{align}
The first and the second terms of the right-hand side are calculated as
\begin{align}
\det(\bu_\zeta, \bu_{\zeta\zeta})^2 - \det(\bv_\zeta, \bv_{\zeta\zeta})^2 
%= & [\det(\bu_\zeta, \bu_{\zeta\zeta}) + \det(\bv_\zeta, \bv_{\zeta\zeta})][\det(\bu_\zeta, \bu_{\zeta\zeta}) - \det(\bv_\zeta, \bv_{\zeta\zeta})] \\
& = [\det(\bu_\zeta, \bu_{\zeta\zeta}) + \det(\bv_\zeta, \bv_{\zeta\zeta})] \\
& \times 
\big[v_{2,\zeta\zeta}(u_{1,\zeta}-v_{1,\zeta}) - v_{1,\zeta\zeta}(u_{2,\zeta}-v_{2,\zeta}) \\
& - u_{2,\zeta}(u_{1,\zeta\zeta}-v_{1,\zeta\zeta}) + u_{1,\zeta}(u_{2,\zeta\zeta}-v_{2,\zeta\zeta}) \big],
\end{align}
and
\begin{align}
\frac{1}{|\bu_\zeta|^5} - \frac{1}{|\bv_\zeta|^5}
&= - \frac{|\bu_\zeta|^{10} - |\bv_\zeta|^{10}}{|\bu_\zeta|^5|\bv_\zeta|^5(|\bu_\zeta|^5 + |\bv_\zeta|^5)} \\
&= - \frac{\sum_{k=0}^{4} |\bu_\zeta|^{8-2k}|\bv_\zeta|^{2k}}{|\bu_\zeta|^5|\bv_\zeta|^5(|\bu_\zeta|^5 + |\bv_\zeta|^5)} (\bu_\zeta + \bv_\zeta)
\cdot (\bu_\zeta - \bv_\zeta),
\end{align}
respectively.
Although we have omitted them, we can derive partial derivatives of $G_1$ with these equations.

Before showing numerical examples, we recall the steady-state solutions for the elastic flow.
It is known that steady closed curves of the elastic energy \eqref{eq:elastic-energy} are circles of radius $\varepsilon$, the figure-eight-shaped curve with scale $\varepsilon$, and their multiple versions \cite{Lov44,Sin08} (see Figure~\ref{fig:steady}).
Their energies are
\begin{equation}
E[\text{circle}] = 4 \pi \varepsilon,
\quad 
E[\text{eight-shaped}] \approx \varepsilon \cdot 21.2075,
\end{equation}
respectively.
The exact value of the latter energy is expressed by the elliptic integrals (cf.~\cite{Sac12}).

\begin{figure}[htb]
\centering
\includegraphics{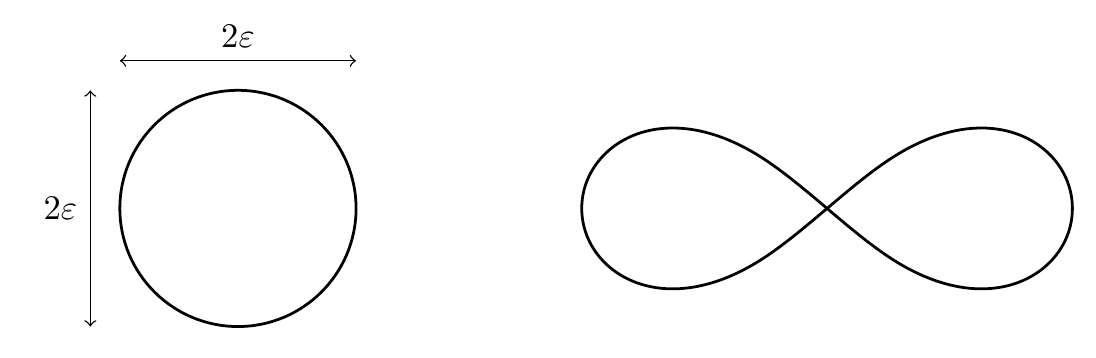}
\caption{Steady states of the elastic energy \protect\eqref{eq:elastic-energy}.}
\label{fig:steady}
\end{figure}

In our numerical examples, we set the time increment as
\begin{equation}
\Delta t_0 = \tau \cdot \min\left\{ 1, \ 100 \left( \int |\bkappa(u_h^0)|^2 ds \right)^{-1} \right\}
\end{equation}
and
\begin{equation}
\Delta t_n = \tau \cdot \min\left\{ 1, \ 100 \left( \frac{E[\bu_h^{n}] - E[\bu_h^{n-1}]}{\Delta t_{n-1}} \right)^{-2} \right\},
\quad n \ge 1.
\end{equation}
We found these values empirically.
We solved equation \eqref{eq:disc-gf} at each step with the Newton method.
Moreover, in each step, when two adjacent control points are too close (more precisely, when the distance is less than $0.1 \cdot \max \{ \varepsilon, 0.05 \}$), one of them was removed.
Note that the energy may increase when control points are eliminated; however, the shape of the curve will be less affected (see Figure~\ref{fig:elim}).

\begin{figure}[htb]
\centering
\includegraphics{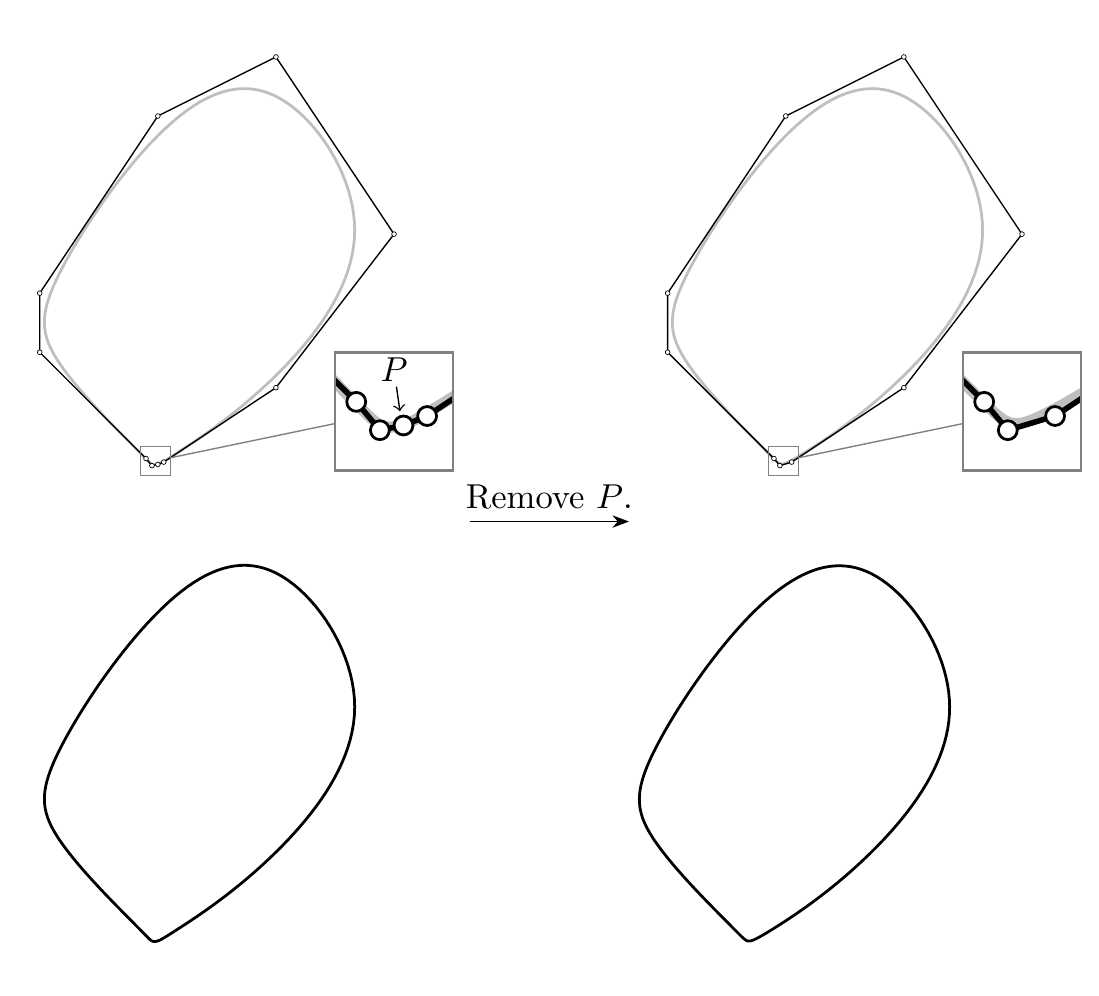}
\caption{Effect of eliminating control points. 
The left two figures show B-spline curves of degree $p=3$ with respect to the control points as shown in the upper figure.
The right two figures show the B-spline curves of the same degree with respect to the control points except for the point $P$.}
\label{fig:elim}
\end{figure}

We show six examples here.
In all examples, we use the B-spline curves of degree $p=3$.
Therefore, every curve below is of class $C^2$.
Videos of the following examples are available on YouTube%
\footnote{URL: \url{https://www.youtube.com/playlist?list=PLMF3dSqWEii6coWnwHHOh12mrtMmIXIi0}}.

\begin{example}\label{ex:circle_0}
The first example is shown in Figure~\ref{fig:circle_0}.
The initial curve is a circle.The parameters are
\begin{equation}
\varepsilon=0.1, \quad N=6, \quad \tau=0.01.
\end{equation}
In this example, the elimination of control points was not necessary.
Figure~\ref{fig:circle_0-curve} shows the evolution of the curve at $t \approx 0, 0.1, \dots,0.6$.
The energy at $t \approx 0.6$ is $E \approx 1.2583$.
Note that the exact value of the energy at the steady state is $4 \pi \varepsilon \approx 1.2566$.
Figure~\ref{fig:circle_0-energy} shows the evolution of the energy.
The discrete energy dissipation property is clearly visible.

In Figure~\ref{fig:circle_0-curve}, one can observe that the curve shrinks as the curvature flow \eqref{eq:curvature} until $t \approx 0.5$, and it stops shrinking when the radius approaches $\varepsilon$.
\end{example}

\begin{figure}[htb]
\centering
\begin{subfigure}[b]{.45\textwidth}
\centering
\includegraphics[page=1]{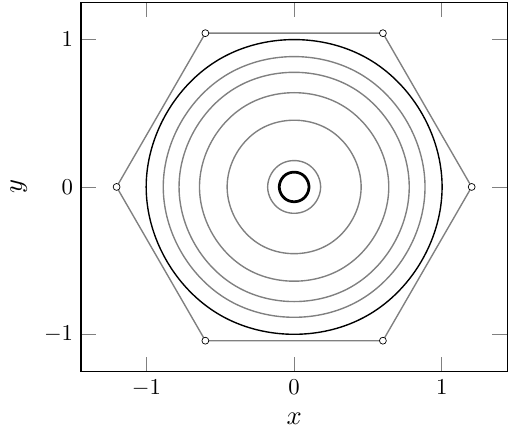}
\caption{Evolution of the circle. 
%The curves are the solution at $t\approx 0,0.1,\dots,0.6$. 
The outermost curve and the innermost one are at times $ t = 0$, and $t \approx 0.6$, respectively.}
\label{fig:circle_0-curve}
\end{subfigure}
\begin{subfigure}[b]{.45\textwidth}
\centering
\includegraphics[page=25]{curves.pdf}
\caption{Evolution of the energy.}
\label{fig:circle_0-energy}
\end{subfigure}
\caption{Example~\ref{ex:circle_0}.}
\label{fig:circle_0}
\end{figure}

\begin{example}\label{ex:eight_0}
The second example is shown in Figure~\ref{fig:eight_0}.
The initial curve is figure-eight-shaped. The parameters are
\begin{equation}
\varepsilon=0.2, \quad N=12, \quad \tau=0.01.
\end{equation}
As in the previous case, the elimination of control points is not necessary.
Figure~\ref{fig:eight_0-curve} shows the evolution of the curve at $t \approx 0, 0.2, \dots,1.2$ and Figure~\ref{fig:eight_0-energy} shows the evolution of the energy.
The energy at $t \approx 1.2$ is $E \approx 4.2433$.
Note that the exact value of the energy at the steady state is approximately $4.2415$.

In Figure~\ref{fig:eight_0-curve}, first the small loop (the right loop) shrinks faster than the larger one.
When the scale of the right loop becomes $\varepsilon$, shrinking stops, and the left one begins to shrink.
Finally, the left one also stops shrinking, and the curve approaches the steady state.
\end{example}

\begin{figure}[htb]
\centering
\begin{subfigure}[b]{.45\textwidth}
\centering
\includegraphics[page=2]{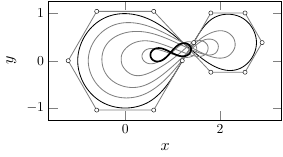}
\caption{Evolution of the figure-eight-shaped curve. 
%The curves are the solution at $t\approx 0,0.1,\dots,0.6$. 
The outermost curve is the initial shape of the curve, and the innermost one is the curve at $t \approx 1.2$.}
\label{fig:eight_0-curve}
\end{subfigure}
\begin{subfigure}[b]{.45\textwidth}
\centering
\includegraphics[page=27]{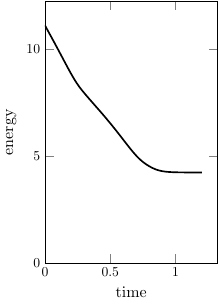}
\caption{Evolution of the energy.}
\label{fig:eight_0-energy}
\end{subfigure}
\caption{Example~\ref{ex:eight_0}.}
\label{fig:eight_0}
\end{figure}

\begin{example}\label{ex:double_0}
The third example is shown in Figure~\ref{fig:double_0-curve}.
The initial shape of the curve is a cardioid-like curve as shown in Figure~\ref{fig:double_0-1}.
The initial parameters of the curve are
\begin{equation}
\varepsilon=0.1, \quad N=12, \quad \tau=0.005.
\end{equation}
As in the previous cases, the elimination of control points is not necessary.
Figure~\ref{fig:double_0-curve} shows the evolution of the curve at $t \approx 0, 0.2, 0.4, 0.6$ and Figure~\ref{fig:double_0-energy} shows the evolution of the energy.
In this case, the steady-state is a double-looped circle with radius $\varepsilon=0.1$.
Therefore, the energy of the solution at $t \approx 0.6$ ($E \approx 2.5228$) is approximately twice the value of that of Example~\ref{ex:circle_0}.

The behavior of the curve is similar to Example~\ref{ex:eight_0}.
That is, first the smaller loop shrinks until the scale is approximately $\varepsilon$.
Then, the larger one shrinks and the curve approaches the steady state.
\end{example}

\begin{figure}[htb]
\begin{minipage}[b]{0.60\textwidth}
\centering
\begin{subfigure}[b]{.49\textwidth}
\centering
\includegraphics[page=3]{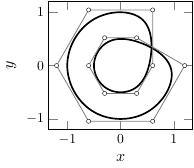}
\caption{$t = 0$.}
\label{fig:double_0-1}
\end{subfigure}
\begin{subfigure}[b]{.49\textwidth}
\centering
\includegraphics[page=4]{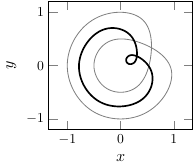}
\caption{$t \approx 0.2$.}
\label{fig:double_0-2}
\end{subfigure}
\\[2ex]
\begin{subfigure}[b]{.49\textwidth}
\centering
\includegraphics[page=5]{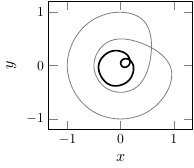}
\caption{$t \approx 0.4$.}
\label{fig:double_0-3}
\end{subfigure}
\begin{subfigure}[b]{.49\textwidth}
\centering
\includegraphics[page=6]{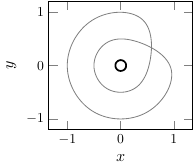}
\caption{$t \approx 0.6$.}
\label{fig:double_0-4}
\end{subfigure}
\caption{Evolution of the curves in Example~\ref{ex:double_0}.}
\label{fig:double_0-curve}
\end{minipage}
\begin{minipage}[b]{0.39\textwidth}
\centering
\includegraphics[page=29]{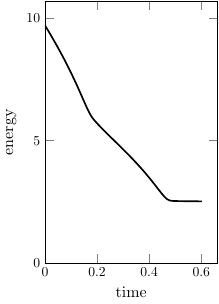}
\caption{Evolution of the energy in Example~\ref{ex:double_0}.}
\label{fig:double_0-energy}
\end{minipage}
\end{figure}

\begin{example}\label{ex:circle_1}
This example shows a topology-changing solution.
The initial curve is the one shown in Figure~\ref{fig:circle_1-1}, and Figure~\ref{fig:circle_1-curve} shows its evolution.
Figures~\ref{fig:circle_1-energy} and \subref{fig:circle_1-numbers} show the evolution of the energy and the number of control points, respectively.
The parameters are
\begin{equation}
\varepsilon=0.2, \quad N=20 \ (\text{initially}), \quad \tau=0.005.
\end{equation}

One can observe that the topology of the curve changes at around $t = 1.05$ (Figures~\ref{fig:circle_1-5} and \subref{fig:circle_1-6}).
At the same time, the energy decreases drastically (Figure~\ref{fig:circle_1-energy}),
and some control points become concentrated (Figure~\ref{fig:circle_1-elim}).
As mentioned earlier, we implement an algorithm that deletes a control point if it is too close to the adjacent point.
Therefore, the elimination of control points occurs when the topology changes, and the number of control points finally converges to $N=11$.
\end{example}

\begin{figure}[htb]
\centering
\begin{subfigure}[b]{.32\textwidth}
\centering
\includegraphics[page=7]{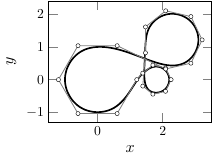}
\caption{$t = 0$.}
\label{fig:circle_1-1}
\end{subfigure}
\begin{subfigure}[b]{.32\textwidth}
\centering
\includegraphics[page=8]{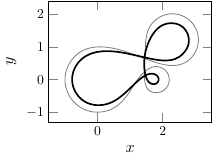}
\caption{$t \approx 0.2$.}
\label{fig:circle_1-2}
\end{subfigure}
\begin{subfigure}[b]{.32\textwidth}
\centering
\includegraphics[page=9]{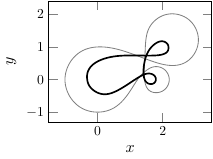}
\caption{$t \approx 0.5$.}
\label{fig:circle_1-3}
\end{subfigure}
\\[2ex]
\begin{subfigure}[b]{.32\textwidth}
\centering
\includegraphics[page=10]{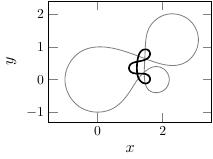}
\caption{$t \approx 1.0$.}
\label{fig:circle_1-4}
\end{subfigure}
\begin{subfigure}[b]{.32\textwidth}
\centering
\includegraphics[page=11]{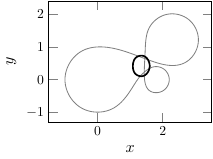}
\caption{$t \approx 1.1$.}
\label{fig:circle_1-5}
\end{subfigure}
\begin{subfigure}[b]{.32\textwidth}
\centering
\includegraphics[page=12]{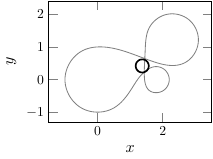}
\caption{$t \approx 1.4$.}
\label{fig:circle_1-6}
\end{subfigure}
\caption{Evolution of the curves in Example~\ref{ex:circle_1}.}
\label{fig:circle_1-curve}
\end{figure}

\begin{figure}[htb]
\centering
\begin{subfigure}[t]{.45\textwidth}
\centering
\includegraphics[page=31]{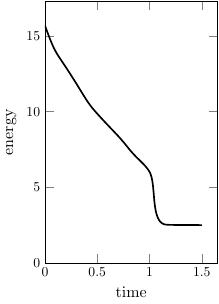}
\caption{Evolution of the energy.}
\label{fig:circle_1-energy}
\end{subfigure}
\begin{subfigure}[t]{.45\textwidth}
\centering
\includegraphics[page=32]{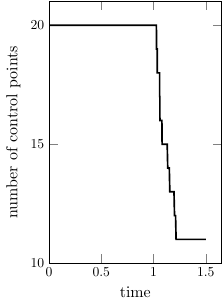}
\caption{Evolution of the number of control points.}
\label{fig:circle_1-numbers}
\end{subfigure}
\caption{Example~\ref{ex:circle_1}.}
\label{fig:circle_1}
\end{figure}

\begin{figure}[htb]
\centering
\begin{subfigure}[b]{.32\textwidth}
\centering
\includegraphics[page=1]{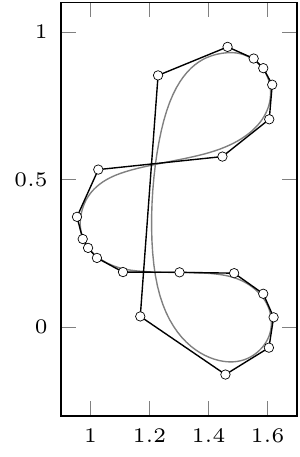}
\caption{$t \approx 1.0$, $N=20$.}
\end{subfigure}
\begin{subfigure}[b]{.32\textwidth}
\centering
\includegraphics[page=2]{elim-circle-1.pdf}
\caption{$t \approx 1.02$, $N=20$.}
\end{subfigure}
\begin{subfigure}[b]{.32\textwidth}
\centering
\includegraphics[page=3]{elim-circle-1.pdf}
\caption{$t \approx 1.03$, $N=19$.}
\end{subfigure}
\\[2ex]
\begin{subfigure}[b]{.32\textwidth}
\centering
\includegraphics[page=4]{elim-circle-1.pdf}
\caption{$t \approx 1.05$, $N=18$.}
\end{subfigure}
\begin{subfigure}[b]{.32\textwidth}
\centering
\includegraphics[page=5]{elim-circle-1.pdf}
\caption{$t \approx 1.06$, $N=17$.}
\end{subfigure}
\begin{subfigure}[b]{.32\textwidth}
\centering
\includegraphics[page=6]{elim-circle-1.pdf}
\caption{$t \approx 1.07$, $N=15$.}
\end{subfigure}
\\[2ex]
\begin{subfigure}[b]{.32\textwidth}
\centering
\includegraphics[page=7]{elim-circle-1.pdf}
\caption{$t \approx 1.15$, $N=14$.}
\end{subfigure}
\begin{subfigure}[b]{.32\textwidth}
\centering
\includegraphics[page=8]{elim-circle-1.pdf}
\caption{$t \approx 1.2$, $N=12$.}
\end{subfigure}
\begin{subfigure}[b]{.32\textwidth}
\centering
\includegraphics[page=9]{elim-circle-1.pdf}
\caption{$t \approx 1.3$, $N=11$.}
\end{subfigure}
\caption{Diminution of control points in Example~\ref{ex:circle_1} at the times described as in the figures.
The small circles represent control points, $N$ is the number of control points, and the gray colored curve is the solution.}
\label{fig:circle_1-elim}
\end{figure}

\begin{example}\label{ex:eight_1}
The following two examples investigate problems with more complicated solutions.
The initial shape of the curve is shown in Figure~\ref{fig:eight_1-1}, and Figure~\ref{fig:eight_1-curve} shows its evolution.
Figures~\ref{fig:eight_1-energy} and \subref{fig:eight_1-numbers} show the evolution of the energy and the number of control points, respectively.
The parameters are
\begin{equation}
\varepsilon=0.2, \quad N=49 \ (\text{initially}), \quad \tau=0.005.
\end{equation}

In this example, the topology of the curve changes frequently.
For example, the loop in the upper left of the curve disappears at around $t=1.35$.
When the topology changes, the energy decreases rapidly as in Example~\ref{ex:circle_1},
and the number of control points also decreases at the same time.
The final value of $N$ is $12$.
\end{example}

\begin{figure}[htb]
\centering
\begin{subfigure}[b]{.32\textwidth}
\centering
\includegraphics[page=13]{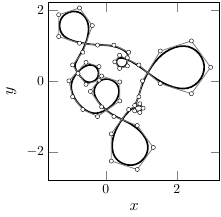}
\caption{$t = 0$.}
\label{fig:eight_1-1}
\end{subfigure}
\begin{subfigure}[b]{.32\textwidth}
\centering
\includegraphics[page=14]{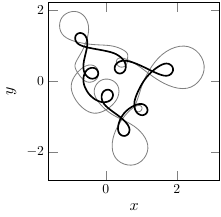}
\caption{$t \approx 0.4$.}
\label{fig:eight_1-2}
\end{subfigure}
\begin{subfigure}[b]{.32\textwidth}
\centering
\includegraphics[page=15]{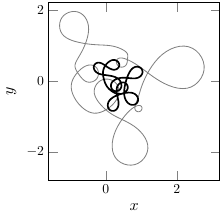}
\caption{$t \approx 1.3$.}
\label{fig:eight_1-3}
\end{subfigure}
\\[2ex]
\begin{subfigure}[b]{.32\textwidth}
\centering
\includegraphics[page=16]{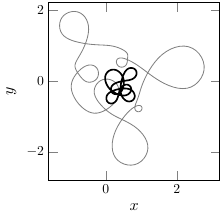}
\caption{$t \approx 1.45$.}
\label{fig:eight_1-4}
\end{subfigure}
\begin{subfigure}[b]{.32\textwidth}
\centering
\includegraphics[page=17]{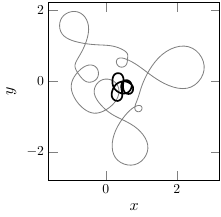}
\caption{$t \approx 1.65$.}
\label{fig:eight_1-5}
\end{subfigure}
\begin{subfigure}[b]{.32\textwidth}
\centering
\includegraphics[page=18]{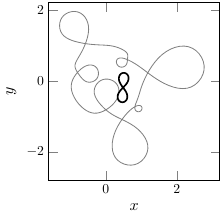}
\caption{$t \approx 2.0$.}
\label{fig:eight_1-6}
\end{subfigure}
\caption{Evolution of the curves in Example~\ref{ex:eight_1}.}
\label{fig:eight_1-curve}
\end{figure}

\begin{figure}[htb]
\centering
\begin{subfigure}[t]{.45\textwidth}
\centering
\includegraphics[page=33]{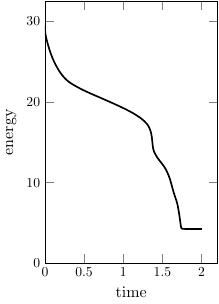}
\caption{Evolution of the energy.}
\label{fig:eight_1-energy}
\end{subfigure}
\begin{subfigure}[t]{.45\textwidth}
\centering
\includegraphics[page=34]{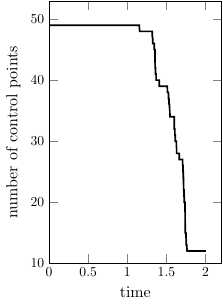}
\caption{Evolution of the number of control points.}
\label{fig:eight_1-numbers}
\end{subfigure}
\caption{Example~\ref{ex:eight_1}.}
\label{fig:eight_1}
\end{figure}

\begin{example}\label{ex:circle_3}
The initial curve for the final example is shown in Figure~\ref{fig:circle_3-1}, and Figure~\ref{fig:circle_3-curve} shows its evolution.
Figures~\ref{fig:circle_3-energy} and \subref{fig:circle_3-numbers} show the evolution of the energy and the number of control points, respectively.
The parameters are
\begin{equation}
\varepsilon=0.2, \quad N=54 \ (\text{initially}), \quad \tau=0.005.
\end{equation}

The solution displays complicated behavior as in Example~\ref{ex:eight_1},
and the topology changes frequently.
However, since the turning number of the initial curve is one, the steady state is a circle with radius $\varepsilon$.
The energy and the number of control points decrease drastically when the topology changes, and the final number of control points is $N=12$.
\end{example}

\begin{figure}[htb]
\centering
\begin{subfigure}[b]{.32\textwidth}
\centering
\includegraphics[page=19]{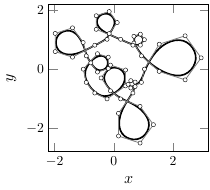}
\caption{$t = 0$.}
\label{fig:circle_3-1}
\end{subfigure}
\begin{subfigure}[b]{.32\textwidth}
\centering
\includegraphics[page=20]{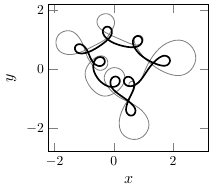}
\caption{$t \approx 0.4$.}
\label{fig:circle_3-2}
\end{subfigure}
\begin{subfigure}[b]{.32\textwidth}
\centering
\includegraphics[page=21]{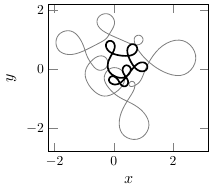}
\caption{$t \approx 1.3$.}
\label{fig:circle_3-3}
\end{subfigure}
\\[2ex]
\begin{subfigure}[b]{.32\textwidth}
\centering
\includegraphics[page=22]{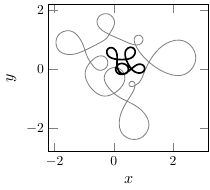}
\caption{$t \approx 1.5$.}
\label{fig:circle_3-4}
\end{subfigure}
\begin{subfigure}[b]{.32\textwidth}
\centering
\includegraphics[page=23]{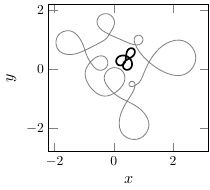}
\caption{$t \approx 1.9$.}
\label{fig:circle_3-5}
\end{subfigure}
\begin{subfigure}[b]{.32\textwidth}
\centering
\includegraphics[page=24]{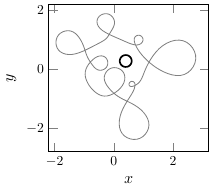}
\caption{$t \approx 2.3$.}
\label{fig:circle_3-6}
\end{subfigure}
\caption{Evolution of the curves in Example~\ref{ex:circle_3}.}
\label{fig:circle_3-curve}
\end{figure}

\begin{figure}[htb]
\centering
\begin{subfigure}[t]{.45\textwidth}
\centering
\includegraphics[page=35]{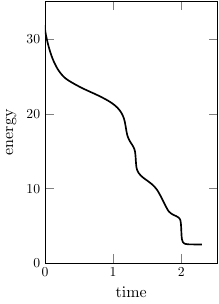}
\caption{Evolution of the energy.}
\label{fig:circle_3-energy}
\end{subfigure}
\begin{subfigure}[t]{.45\textwidth}
\centering
\includegraphics[page=36]{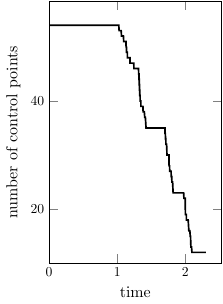}
\caption{Evolution of the number of control points.}
\label{fig:circle_3-numbers}
\end{subfigure}
\caption{Example~\ref{ex:circle_3}.}
\label{fig:circle_3}
\end{figure}

\section*{Acknowledgements}
I would like to thank Prof.~Yoshihiro Tonegawa and Dr.~Takahito Kashiwabara for bringing this topic to my attention and encouraging me through valuable discussions.
This work was supported by the Program for Leading Graduate Schools, MEXT,
Japan, and by JSPS KAKENHI (no.~15J07471).

\bibliographystyle{plain}
%%\bibliographystyle{alpha}
%\bibliography{elastic-flow}

\end{document}